\documentclass{amsart}

\usepackage{hyperref}

\usepackage{amssymb}
\usepackage{amsmath}
\usepackage{amsthm}
\usepackage{enumerate}
\usepackage{graphicx}
\usepackage{color}

\theoremstyle{plain}

\newtheorem{theorem}{Theorem}[section]

\newtheorem{proposition}[theorem]{Proposition}

\newtheorem{lemma}[theorem]{Lemma}
\newtheorem{corollary}[theorem]{Corollary}
\newtheorem{question}[theorem]{Question}

\theoremstyle{definition}
\newtheorem{definition}[theorem]{Definition}

\theoremstyle{remark}

\newtheorem*{remarks}{Remarks}

\newcommand{\RiemannSphere}{\widehat{\mathbb{C}}}
\newcommand{\cS}{{\mathcal H^\uparrow}}

\begin{document}

\title{Continuity of capacity of a holomorphic motion}

%\date{March 24, 2020}

\author[T. Ransford]{Thomas Ransford}
\address{D\'epartement de math\'ematiques et de statistique, Universit\'e Laval, Qu\'ebec (Qu\'ebec),  G1V 0A6, Canada.}
\email{thomas.ransford@mat.ulaval.ca}

\author[M. Younsi]{Malik Younsi}
\address{Department of Mathematics, University of Hawaii Manoa, Honolulu, HI 96822, USA.}
\email{malik.younsi@gmail.com}

\author[W.-H. Ai]{Wen-hui Ai}
\address{D\'epartement de math\'ematiques et de statistique, Universit\'e Laval, Qu\'ebec (Qu\'ebec),  G1V 0A6, Canada.}
\curraddr{College of Mathematics and Econometrics, Hunan University, Changsha, 410082, People'€™s Republic of China.}
\email{awhxyz123@163.com}

\thanks{The first author was supported by grants from NSERC and the Canada Research Chairs program.
The second author was supported by NSF Grant DMS-1758295. The third author was supported in part by the NNSF of China (No.\ 11831007) and China Scholarship Council (No.\ 201906130015).}

\begin{abstract}
We study the behavior of various set-functions under holomorphic motions. We show that, under such deformations, logarithmic capacity varies continuously, while  analytic capacity may not.
\end{abstract}

\maketitle

\section{Introduction and statement of main results}
\label{sec1}

\subsection{Holomorphic motions}

Let $\mathbb{D}$ be the open unit disk in the complex plane~$\mathbb{C}$. Given a subset $A$ of the Riemann sphere $\RiemannSphere$, a \textit{holomorphic motion} of $A$ is a map $h:\mathbb{D} \times A \to \RiemannSphere$ such that

\begin{enumerate}[(i)]
\item for each fixed $z\in A$, the map $\lambda\mapsto h(\lambda,z)$
is holomorphic on $\mathbb{D}$,
\item for each fixed $\lambda \in \mathbb{D}$, the map $z\mapsto h(\lambda,z)$ is
injective on $A$,
\item $h(0,z)=z$ for all $z\in A$.
\end{enumerate}

For a holomorphic motion $h:\mathbb{D} \times A \to \RiemannSphere$ and a subset $E \subset A$, we write
$$h_\lambda(z):=h(\lambda,z) \qquad (\lambda \in \mathbb{D}, z\in A)$$
and
$$E_\lambda:=h_\lambda(E)$$
so that $E_0=E$.

Holomorphic motions were introduced by Ma\~{n}\'{e}, Sad and Sullivan \cite{MSS}, who proved that every holomorphic motion $h:\mathbb{D} \times A \to \RiemannSphere$ has an extension to a holomorphic motion $H:\mathbb{D} \times \overline{A} \to \RiemannSphere$, and that $H$ is jointly continuous in $(\lambda,z)$. This is usually referred to as the $\lambda$-lemma.

Sullivan and Thurston later conjectured in \cite{ST} that every holomorphic motion extends to a holomorphic motion of the whole Riemann sphere $\RiemannSphere$, and proved partial results. The full conjecture was finally proved by S\l odkowski in 1991.

\begin{theorem}[S\l odkowski \cite{SLO}]
\label{thmSlod}
If $h:\mathbb{D} \times A \to \RiemannSphere$ is a holomorphic motion of a set $A \subset \RiemannSphere$, then $h$ has an extension to a holomorphic motion $H: \mathbb{D} \times \RiemannSphere \to \RiemannSphere$.
\end{theorem}

Theorem \ref{thmSlod} serves as motivation to consider holomorphic motions of the whole Riemann sphere. Now, consider a real-valued function $\mathcal{F}$ defined on compact subsets of $\mathbb{C}$. Let $E \subset \mathbb{C}$ be compact, and suppose that $h:\mathbb{D} \times \RiemannSphere \to \RiemannSphere$ is a holomorphic motion of $\RiemannSphere$ with $h_\lambda(\infty)=\infty$ for all $\lambda \in \mathbb{D}$. Then, for each $\lambda \in \mathbb{D}$, the set $E_\lambda \subset \mathbb{C}$ is compact by the $\lambda$-lemma, so the function
$$\lambda \mapsto \mathcal{F}(E_\lambda) \qquad (\lambda \in \mathbb{D})$$
is well-defined. The behavior of this function has been extensively studied, for various set-functions $\mathcal{F}$. For instance, the behavior of Hausdorff dimension under holomorphic motions was studied by Ruelle \cite{RUE} and by Astala--Zinsmeister \cite{AZ} in the setting of limit sets of Fuschian groups, and by Ruelle \cite{RUE} and by the first author \cite{RAN} for Julia sets of rational maps. Holomorphic motions also played a fundamental role in the seminal work of Astala in \cite{AST} on distortion of dimension and area under quasiconformal mappings. Lastly, we also mention the work of Earle--Mitra \cite{EM} on conformal modulus and the recent work of Zakeri \cite{ZAK} on applications to holomorphic dynamics.

Given a compact subset $E$ of $\mathbb{C}$, we denote by $A(E)$ the area of $E$, and by $\dim_H(E)$ the Hausdorff dimension of $E$. As an easy consequence of Astala's results in \cite{AST}, we obtain the following result.

\begin{theorem}[Astala \cite{AST}]
\label{Astala}
Let $E \subset \mathbb{C}$ be compact and let $h:\mathbb{D} \times \RiemannSphere \to \RiemannSphere$ be a holomorphic motion of $\RiemannSphere$ with $h_\lambda(\infty)=\infty$ for all $\lambda \in \mathbb{D}$. Then the functions
$$\lambda \mapsto A(E_\lambda)
\quad\text{and}\quad
\lambda\mapsto\dim_H(E_\lambda) \qquad (\lambda \in \mathbb{D})$$
are both continuous.
\end{theorem}

Our aim in this article is to determine whether similar results hold for logarithmic capacity and analytic capacity.

\subsection{Logarithmic capacity and holomorphic motions}

Recall that the \textit{logarithmic capacity} of a compact set $E \subset \mathbb{C}$ is defined by
$$c(E):=\exp{\Bigl(-\inf_\mu I(\mu)\Bigr)},$$
where the infimum is taken over all Borel probability measures $\mu$ supported on $E$ and $I(\mu):=-\int \int \log{|z-w|} \, d\mu(z) \, d\mu(w)$ is the logarithmic energy of the measure $\mu$.

Our first main result is the following theorem.

\begin{theorem}
\label{mainthm1}
Let $E \subset \mathbb{C}$ be compact and let $h:\mathbb{D} \times \RiemannSphere \to \RiemannSphere$ be a holomorphic motion of $\RiemannSphere$ with $h_\lambda(\infty)=\infty$ for all $\lambda \in \mathbb{D}$. Then the function
$$\lambda \mapsto c(E_\lambda) \qquad (\lambda \in \mathbb{D})$$
is  continuous.
\end{theorem}

The proof of Theorem \ref{mainthm1} is based on an alternative characterization of logarithmic capacity as the transfinite diameter. It also shows that, under the same hypotheses, $\lambda\mapsto\log c(E_\lambda)$ is a subharmonic function.
This fact was already known; it is a special case of a result of Yamaguchi on analytic multifunctions \cite{YA}.

\subsection{Analytic capacity and holomorphic motions}

Recall that the \textit{analytic capacity} of a compact set $K \subset \mathbb{C}$ is defined by
$$\gamma(K):=\sup_{g}|g'(\infty)|,$$
where the supremum is taken over all holomorphic functions $g:\RiemannSphere \setminus K \to \mathbb{D}.$ Here $g'(\infty)$ is defined by
$$
g'(\infty):=\lim_{z\to\infty}z(g(z)-g(\infty)).
$$

Analytic capacity was introduced by Ahlfors \cite{AHL} for the study of the Painlev\'e problem of characterizing removable singularities for bounded holomorphic functions. It was later observed by Vitushkin \cite{VIT} that analytic capacity is a fundamental tool in the theory of rational approximation of holomorphic functions. Analytic capacity is notoriously hard to estimate and its properties remain quite mysterious, although there exist efficient numerical methods in some cases, see e.g.\ \cite{YOR}. For more information on analytic capacity, we refer the reader to \cite{GAR}, \cite{TOL} and \cite{YOU}.

The study of the behavior of analytic capacity under holomorphic motions was instigated by the first two authors and Pouliasis in \cite{PRY}, who showed that there exist a compact set $E$ and a holomorphic motion $h:\mathbb{D} \times \RiemannSphere \to \RiemannSphere$ fixing $\infty$ for which the functions $\lambda \mapsto \gamma(E_\lambda)$ and $\lambda \mapsto \log{\gamma(E_\lambda)}$ are neither subharmonic nor superharmonic on $\mathbb{D}$.
It turns out that these functions need not be continuous either. This is our second main result.

\begin{theorem}
\label{mainthm2}
There exist a compact set $E \subset \mathbb{C}$ and a holomorphic motion $h:\mathbb{D} \times \RiemannSphere \to \RiemannSphere$ such that $h_\lambda(\infty)=\infty$ for all $\lambda \in \mathbb{D}$, but the function
$$\lambda \mapsto \gamma(E_\lambda) \qquad (\lambda \in \mathbb{D})$$
is not continuous.
\end{theorem}

Some remarks are in order here. First, note that analytic capacity is closely related to $1$-dimensional Hausdorff measure $\mathcal{H}^1$. Indeed, a classical result generally attributed to Painlev\'e states that $\gamma(K)=0$ whenever $\mathcal{H}^1(K)=0$. The converse is false in general, although it does hold provided that the set $K$ is contained in a smooth (even rectifiable) curve. In view of this, a natural approach to proving Theorem~\ref{mainthm2} is first to study the behavior of $1$-dimensional Hausdorff measure under holomorphic motions. It turns out that one can easily construct a compact set $E$ and a holomorphic motion $h:\mathbb{D} \times \RiemannSphere \to \RiemannSphere$ for which $\mathcal{H}^1(E_\lambda)$ does not vary continuously with $\lambda$; we present such an example in Section~\ref{S:conclusion}. The sets $E_\lambda$ of this example, however, are dynamical in nature, and it is unclear how to determine or even estimate their analytic capacity.

In order to circumvent this difficulty, we instead construct an example for which the sets $E_\lambda$ lie on a smooth curve, in fact, on the unit circle $\mathbb{T}$.
The idea is to consider a special collection of quasiconformal mappings $h_\lambda: \RiemannSphere \to \RiemannSphere$ depending holomorphically on $\lambda$. For $\lambda \in \mathbb{D} \cap \mathbb{R}$, the map $h_\lambda$ is constructed so that its restriction to the unit circle $\mathbb{T}$ is the conformal welding of a quasicircle quadratic Julia set with corresponding parameter in the main cardioid of the Mandelbrot set. A classical result of Fatou combined with a theorem of Bishop--Carleson--Garnett--Jones imply that the harmonic measures from both sides of this Julia set are mutually singular. We deduce from this that $h_\lambda$ is singular, i.e. it maps some Borel subset $A^\lambda$ of $\mathbb{T}$ with full Lebesgue measure onto a subset of $\mathbb{T}$ with zero Lebesgue measure. Theorem \ref{mainthm2} then  follows by combining this construction with a classical result of Murai on the analytic capacity of subsets of $\mathbb{T}$.

The rest of the paper is organized as follows. Section~\ref{logcap} contains the proof of Theorem~\ref{mainthm1}, together with a summary of the necessary prerequisites from potential theory.
Section~\ref{ancap} contains the proof of Theorem~\ref{mainthm2},
including the necessary preliminaries on quasiconformal mappings, Julia sets and conformal welding. We conclude in Section~\ref{S:conclusion} with some remarks and a question.

\section{Logarithmic capacity and the proof of Theorem~\ref{mainthm1}}
\label{logcap}

\subsection{Transfinite diameter}
Let $E \subset \mathbb{C}$ be compact. Recall from the introduction that the logarithmic capacity of $E$ is given by
$$c(E):=\exp{\Bigl(-\inf_\mu I(\mu)\Bigr)},$$
where the infimum is taken over all Borel probability measures $\mu$ supported on $E$ and $I(\mu)=-\int \int \log{|z-w|} \, d\mu(z) \, d\mu(w)$.
To establish Theorem~\ref{mainthm1}, it is more convenient to work with an alternative characterization of logarithmic capacity.

\begin{definition}
For $E \subset \mathbb{C}$ compact and $n \geq 2$, we define the \textit{n-th diameter} of $E$ by
$$\delta_n(E):= \sup \left\{ \prod_{j<k} |w_j-w_k|^{2/n(n-1)}:w_1, \dots, w_n \in E \right\}.$$
\end{definition}

\begin{theorem}[Fekete--Szeg\H o]
\label{Fekete}
The sequence $(\delta_n(E))_{n \geq 2}$ is decreasing, and
$$\lim_{n \to \infty} \delta_n(E)=c(E).$$
\end{theorem}
See \cite[Theorem 5.5.2]{RAN2}. The limit is often called the \emph{transfinite diameter} of $E$.

\subsection{Pointwise suprema of harmonic functions}

We now introduce a new family of functions.

\begin{definition}
Let $\Omega \subset \mathbb{C} $ be a domain.
We define $\cS(\Omega)$ to be the family of functions $u:\Omega \to [-\infty,\infty)$ such that:
\begin{enumerate}[(a)]
\item $u$ is locally bounded above on $\Omega$,
\item $u$ is the pointwise supremum of a family of harmonic functions on $\Omega$.
\end{enumerate}
\end{definition}

\begin{remarks}
\begin{enumerate}[(i)]
\item We allow the family in (b) to be empty; in this case $u \equiv -\infty$.
\item The supremum in (b) is attained at every point. This easily follows from a normal-families argument.
\item It can happen that $u$ satisfies condition (b) without satisfying condition (a). One can construct examples
using Runge's theorem.
\end{enumerate}
\end{remarks}

The following result summarizes some basic stability properties of $\cS(\Omega)$.

\begin{proposition}
\label{propstab}
Let $\Omega$ be a plane domain. Then
\begin{enumerate}[\normalfont(i)]
\item If $u,v \in \cS(\Omega)$ and $\alpha, \beta \geq 0$, then $\alpha u + \beta v \in \cS(\Omega)$.
\item If $u_i \in \cS(\Omega)$ for all $i\in I$ and if $u:=\sup_{i\in I} u_i$ is locally bounded above, then $u \in \cS(\Omega)$.
\item If $u_n \in \cS(\Omega)$ for each $n \in \mathbb{N}$ and if $u_n\downarrow u$, then $u \in \cS(\Omega)$.
\end{enumerate}
\end{proposition}

\begin{proof}
Parts~(i) and (ii) are obvious. Part~(iii) is also clear if $u \equiv -\infty$. In the remaining case, fix a point $\lambda_0 \in \Omega$ such that $u(\lambda_0)>-\infty$, and, for each $n \in \mathbb{N}$, choose a
harmonic function $h_n$ on $\Omega$ such that $h_n \leq u_n$ and $h_n(\lambda_0)=u_n(\lambda_0)$. By a normal-families argument, a subsequence of the $h_n$ converges locally uniformly to a function $h$ harmonic on $\Omega$ such that $h \leq u$ and $h(\lambda_0)=u(\lambda_0)$. This shows that $u \in \cS(\Omega)$.
\end{proof}

The next result summarizes some properties of elements of $\cS(\Omega)$.
We write $\tau_\Omega$ for the Harnack metric on $\Omega$, defined by
$$
\tau_\Omega(\lambda,\mu):=\sup_k \frac{k(\lambda)}{k(\mu)}
\qquad(\lambda,\mu\in \Omega),
$$
where the supremum is taken over all positive harmonic functions $k$ on $\Omega$.

\begin{proposition}
\label{prop1}
Let $\Omega$ be a plane domain, let $u \in \cS(\Omega)$, and suppose that $u \not\equiv -\infty$. Then
\begin{enumerate}[\normalfont(i)]
\item $u(\lambda)>-\infty$ for all $\lambda\in \Omega$.
\item If  $u< M$ on $\Omega$, then
$$
\frac{M-u(\lambda)}{M-u(\mu)} \leq \tau_\Omega(\lambda,\mu) \qquad (\lambda, \mu \in \Omega).
$$
\item $u$ is a continuous subharmonic function on $\Omega$.
\end{enumerate}
\end{proposition}

\begin{proof}
Part (i) is obvious.

For (ii), note that if $h$ is harmonic and $h \leq u$, then $M-h$ is a positive harmonic function, so it satisfies
$$
\frac{M-h(\lambda)}{M-h(\mu)} \leq \tau_\Omega(\lambda,\mu) \qquad(\lambda,\mu\in \Omega).
$$
As $u$ is the pointwise supremum of all such $h$, it satisfies the same inequality.

For (iii), applying part~(ii) on a compact subdisk of $\Omega$, we see that $u$ is continuous there, since $\tau_\Omega(\lambda,\mu) \to 1$ as $\lambda \to \mu$, by \cite[Theorem 1.3.8]{RAN2}. It is clear that $u$ satisfies the submean property, so it is subharmonic on $\Omega$.
\end{proof}

\subsection{Proof of Theorem~\ref{mainthm1}}

\begin{theorem}
\label{thmcap}
Let $E \subset \mathbb{C}$ be compact and let $h:\mathbb{D} \times \RiemannSphere \to \RiemannSphere$ be a holomorphic motion with $h_\lambda(\infty)=\infty$ for all $\lambda \in \mathbb{D}$. Then the function
$$\lambda \mapsto \log c(E_\lambda) \qquad (\lambda \in \mathbb{D})$$
belongs to $\cS(\mathbb{D})$.
\end{theorem}

\begin{proof}
If $E$ is a finite set, then so is $E_\lambda$ for all $\lambda$, and $\log c(E_\lambda)\equiv-\infty$.
So we may assume at the outset that $E$ is infinite.

Let $n\ge2$ and let $z_1,\dots,z_n$ be distinct points of $E$. Then the map
$$
\lambda \mapsto \prod_{\substack{j<k}} \left(h_\lambda(z_j)-h_\lambda(z_k)\right)
$$
is holomorphic and non-vanishing on $\mathbb{D}$, so the logarithm of its modulus is harmonic there.
It follows that the function
$$\lambda \mapsto \log \delta_n(E_\lambda) \qquad (\lambda \in \mathbb{D})$$
is the supremum of a family of harmonic functions on $\mathbb{D}$. It is also locally bounded above on $\mathbb{D}$,
since $h$ is jointly continuous in $(\lambda,z)$, by the $\lambda$-lemma. Therefore the function
$$ \lambda \mapsto \log \delta_n(E_\lambda) \qquad (\lambda \in \mathbb{D})$$
belongs to $\cS(\mathbb{D})$. Since $\delta_n(E_\lambda)\downarrow c(E_\lambda)$ by Theorem \ref{Fekete}, we deduce that the function
$$ \lambda \mapsto \log c(E_\lambda) \qquad (\lambda \in \mathbb{D})$$
also belongs to $\cS(\mathbb{D})$, by part (iii) of Proposition \ref{propstab}.
\end{proof}

\begin{proof}[Completion of the proof of Theorem \ref{mainthm1}]
By Theorem \ref{thmcap}, the function
$$ \lambda \mapsto \log c(E_\lambda) \qquad (\lambda \in \mathbb{D})$$
belongs to $\cS(\mathbb{D})$.
Hence, either it is identically $-\infty$,
or it is everywhere finite, and then it is continuous and subharmonic on $\mathbb{D}$,
by part (iii) of Proposition \ref{prop1}. In all cases, the function
$$ \lambda \mapsto c(E_\lambda) \qquad (\lambda \in \mathbb{D})$$
is continuous, as required.
\end{proof}

%%%%%%%%%%%%%%%%%%%%%%%%%%

\section{Analytic capacity and the proof of Theorem~\ref{mainthm2}}
\label{ancap}

The proof of Theorem~\ref{mainthm2} uses ideas from dynamical systems and the theory of quasiconformal mappings. We begin by describing the necessary background.

\subsection{Quasiconformal mappings}

\begin{definition}
Let $U,V$ be domains in $\RiemannSphere$. An orientation-preserving homeomorphism $f:U \to V$ is \textit{quasiconformal} if $f$ belongs to the Sobolev space $W^{1,2}_{loc}(U)$ and satisfies the Beltrami equation
$$
\frac{\partial f}{\partial \overline{z}}=\mu \, \frac{\partial f}{\partial z}
$$
almost everywhere, for some measurable function $\mu$ on $U$ with $\|\mu\|_{\infty} <1$ (called the \textit{Beltrami coefficient} of $f$).
\end{definition}
A quasiconformal mapping $f:U \to V$ is conformal if and only if its Beltrami coefficient is zero almost everywhere on $U$. This is usually referred to as Weyl's lemma. We will also need the following fundamental result on the existence and uniqueness of solutions to the Beltrami equation.

\begin{theorem}[Measurable Riemann mapping theorem]
\label{MRMT}
Let $U$ be a domain in $\RiemannSphere$ and let $\mu:U \to \mathbb{D}$ be a measurable function with $\|\mu\|_\infty <1$. Then there exists a quasiconformal mapping $f$ on $U$ such that
$$
\frac{\partial f}{\partial \overline{z}}=\mu \, \frac{\partial f}{\partial z}
$$
almost everywhere. Moreover, the map $f$ is unique up to postcomposition by a conformal map.
\end{theorem}

For more information on quasiconformal mappings, we refer the reader to \cite{LEH}.

Recall from Theorem \ref{thmSlod} that every holomorphic motion extends to a holomorphic motion $h:\mathbb{D} \times \RiemannSphere \to \RiemannSphere$ of the whole sphere, and that $h$ is jointly continuous in $(\lambda,z)$, by the $\lambda$-lemma. In particular, the maps $h_\lambda:\RiemannSphere \to \RiemannSphere$ are homeomorphisms. In fact, much more is true: the $h_\lambda$ are even quasiconformal. There is a sense in which this tells the complete story. This is made precise by the following result.

\begin{theorem}[Bers--Royden \cite{BR}]
\label{thmBR}
Let $h:\mathbb{D} \times \RiemannSphere \to \RiemannSphere$. The following are equivalent:
\begin{enumerate}[\normalfont(i)]
\item The map $h$ is a holomorphic motion.
\item For each $\lambda \in \mathbb{D}$, the map $h_\lambda : \RiemannSphere \to \RiemannSphere$ is quasiconformal with Beltrami coefficient $\mu_\lambda$ satisfying $\|\mu_\lambda\|_\infty \leq |\lambda|$. Moreover, the map $h_0$ is the identity, and the $L^\infty(\mathbb{C})$-valued map $\lambda \mapsto \mu_\lambda$ is holomorphic on $\mathbb{D}$.
\end{enumerate}
\end{theorem}

The converse implication in Theorem \ref{thmBR} actually follows from the fact that solutions of the Beltrami equation depend analytically on the parameter, see e.g. \cite[Section 5.7]{AIM}. For more information on holomorphic motions, we refer the reader to \cite{AM} and \cite[Chapter 12]{AIM}.

\subsection{Quadratic Julia sets}\label{S:Julia}
For a quadratic polynomial $p_c(z):=z^2+c$, we define the \textit{filled Julia set} $\mathcal{K}_c$ of $p_c$ by
$$\mathcal{K}_c:=\{z \in \mathbb{C} : p_c^m(z) \nrightarrow \infty \,\, \mbox{as}\,\, m \to \infty\},$$
where $p_c^m$ denotes the $m$-th iterate of the polynomial $p_c$. The \textit{Julia set} $\mathcal{J}_c$ of $p_c$ is defined as the boundary of the filled Julia set:
$$\mathcal{J}_c:=\partial \mathcal{K}_c,$$
and the \textit{Mandelbrot set} $\mathcal{M}$ is defined as the set of all parameters $c \in \mathbb{C}$ for which the orbit of $0$ under $p_c$ remains bounded:
$$\mathcal{M}:= \{c \in \mathbb{C} : p_c^m(0) \nrightarrow \infty \,\, \mbox{as}\,\, m \to \infty\}.$$
Note that a parameter $c$ belongs to $\mathcal{M}$ if and only if the corresponding Julia set $\mathcal{J}_c$ is connected. For $c \notin \mathcal{M}$, the Julia set $\mathcal{J}_c$ is a Cantor set.
For background on Julia sets and the Mandelbrot set, the reader may consult \cite{MCM} and \cite{MIL}.

We will be mainly interested in the \textit{main cardioid} $\mathcal{M}_0$ of the Mandelbrot set, defined as the set of all parameters $c$ for which the polynomial $p_c$ has an attracting fixed point. It is easy to see that $\mathcal{M}_0$ contains $\mathbb{D}(0,1/4)$, the open disk centered at the origin of radius $1/4$.
For each $c\in \mathcal{M}_0$, the Julia set $\mathcal{J}_c$ is a quasicircle.

It is well known that quadratic Julia sets with parameter belonging to the main cardioid $\mathcal{M}_0$ move holomorphically. More precisely, for each $c \in \mathcal{M}_0$, there is a unique conformal map $B_c:\RiemannSphere \setminus \overline{\mathbb{D}} \to \RiemannSphere\setminus \mathcal{K}_c$ with normalization $B_c(z) = z + O(1/z)$ at $\infty$, called the \textit{B\"{o}ttcher map}. The map $B(c,z):=B_c(z)$ is holomorphic in both variables. Note that $B_0$ is the identity since $\mathcal{J}_0$ is the unit circle.

In order to work on the unit disk rather than $\mathbb{D}(0,1/4)$, we make the change of variable $c=\lambda/4$.
For $\lambda\in{\mathbb D}$, we denote by $\Omega_\lambda^-$ and $\Omega_\lambda^+$
the bounded and unbounded components of $\RiemannSphere \setminus \mathcal{J}_{\lambda/4}$ respectively.
Then the map $g: \mathbb{D} \times (\RiemannSphere \setminus \overline{\mathbb{D}}) \to \RiemannSphere$ defined by $g(\lambda,z):=B_{\lambda/4}(z)$ gives a holomorphic motion of $\RiemannSphere \setminus \overline{\mathbb{D}}$. By Theorems~\ref{thmSlod} and  \ref{thmBR}, this holomorphic motion extends to a holomorphic motion $g:\mathbb{D} \times \RiemannSphere \to \RiemannSphere$ of the whole sphere and, for each $\lambda \in \mathbb{D}$, the map $g_\lambda$ is a quasiconformal homeomorphism of $\RiemannSphere$ which maps $\RiemannSphere \setminus \overline{\mathbb{D}}$ conformally onto $\Omega_{\lambda}^+$.

We will also need the following classical result of Fatou on tangent points of quasicircle Julia sets.

\begin{theorem}[Fatou]
\label{TheoremFatou}
There exists $\lambda_0$ such that for all $\lambda \in \mathbb{D}$ with $0<|\lambda|<\lambda_0$, the corresponding quasicircle Julia set $\mathcal{J}_{\lambda/4}$ has no tangent point.
\end{theorem}

See \cite[Chapter 5, Section 3, Theorem 1]{STE}. We also mention the related result of Hamilton in \cite[Theorem 2]{HAM}, who proved that the set of tangent points of a Jordan curve rational Julia set always has zero $1$-dimensional Hausdorff measure, unless the Julia set is a circle.

\subsection{Singular harmonic measures}\label{S:sing}

For a Jordan curve $\Gamma$ in $\mathbb{C}$ separating $0$ from~$\infty$, denote by $\Omega^-$ and $\Omega^+$ the bounded and unbounded components of $\RiemannSphere \setminus \Gamma$ respectively. Let $\omega^-, \omega^+$ be the harmonic measures of the domains $\Omega^-,\Omega^+$ with respect to $0,\infty$ respectively. The following result of Bishop--Carleson--Garnett--Jones gives a geometric necessary and sufficient condition for the harmonic measures $\omega^-$ and $\omega^+$ to be mutually singular.

\begin{theorem}[Bishop--Carleson--Garnett--Jones \cite{BCGJ}]
\label{TheoremBCGJ}
The harmonic measures $\omega^-$ and $\omega^+$ are mutually singular if and only if the set of tangent points of $\Gamma$ has zero $1$-dimensional Hausdorff measure.
\end{theorem}

Combining Theorem \ref{TheoremFatou} with Theorem \ref{TheoremBCGJ}, we deduce that, for all $\lambda \in \mathbb{D}$ with $0<|\lambda|<\lambda_0$, the corresponding harmonic measures $\omega_{\lambda}^-$ and $\omega_{\lambda}^+$ are mutually singular. Here $\omega_{\lambda}^-, \omega_{\lambda}^+$ denote the harmonic measures of the domains $\Omega_{\lambda}^-,\Omega_{\lambda}^+$ with respect to $0,\infty$ respectively.

\subsection{Conformal welding}
As mentioned in the introduction, the construction in Theorem \ref{mainthm2} is based on the notion of conformal welding. Let $\Gamma \subset \mathbb{C}$ be a Jordan curve and, as before, denote by $\Omega^-$ and $\Omega^+$ the bounded and unbounded components of $\RiemannSphere \setminus \Gamma$ respectively. If $f:\mathbb{D} \to \Omega^-$ and $g:\RiemannSphere \setminus \overline{\mathbb{D}} \to \Omega^+$ are conformal maps, then both $f$ and $g$ extend to homeomorphisms on the closure of their respective domains. The composition $f^{-1} \circ g : \mathbb{T} \to \mathbb{T}$ defines a homeomorphism of the unit circle onto itself called a \textit{conformal welding} of the curve $\Gamma$. It is uniquenely determined by $\Gamma$ up to precomposition and postcomposition by automorphisms of~$\mathbb{D}$.

\subsection{Proof of Theorem~\ref{mainthm2}}
We prove Theorem \ref{mainthm2} by constructing a compact set $E \subset \mathbb{C}$ and a holomorphic motion $h:\mathbb{D} \times \RiemannSphere \to \RiemannSphere$ such that $h_\lambda(\infty)=\infty$, but the function
$$\lambda \mapsto \gamma(E_\lambda) \qquad (\lambda \in \mathbb{D})$$
is not continuous. We shall in fact construct $E$ and $h$ such that $E_\lambda \subset \mathbb{T}$ for all $\lambda \in \mathbb{D} \cap \mathbb{R}$.

Recall from \S\ref{S:Julia} above that the B\"{o}ttcher maps give a holomorphic motion $g: \mathbb{D} \times \RiemannSphere \to \RiemannSphere$ such that, for each $\lambda \in \mathbb{D}$, the map $g_\lambda(z)=g(\lambda,z)$ is a quasiconformal mapping of $\RiemannSphere$ which maps $\RiemannSphere \setminus \overline{\mathbb{D}}$ conformally onto $\Omega_{\lambda}^+$, with normalization $g_\lambda(z) = z + O(1/z)$ at $\infty$. Here $\Omega_{\lambda}^-, \Omega_{\lambda}^+$ denote the bounded and unbounded components of $\RiemannSphere \setminus \mathcal{J}_{\lambda/4}$ respectively.

Now, for $\lambda \in \mathbb{D}$, let $\nu_\lambda$ be the Beltrami coefficient of the quasiconformal mapping $g_\lambda$, and define another Beltrami coefficient $\mu_\lambda$ on $\RiemannSphere$ by
\[
\mu_\lambda(z) :=
\begin{cases}
\nu_\lambda(z), & \textrm{if $|z|<1$},\\
0, & \textrm{if $|z|=1$},\\
\overline{\nu_{\overline{\lambda}} \left(\frac{1}{\overline{z}} \right)} \frac{z^4}{|z|^4}, & \textrm{if $|z|>1$}.
\end{cases}
\]
By the measurable Riemann mapping theorem (Theorem \ref{MRMT}), there is a unique quasiconformal mapping $h_\lambda : \RiemannSphere \to \RiemannSphere$ fixing $0,1,\infty$ whose Beltrami coefficient is equal to $\mu_\lambda$ almost everywhere on $\RiemannSphere$. Note that $h_0$ is the identity.

\begin{proposition}
\label{confweld}
The maps $h_\lambda: \RiemannSphere \to \RiemannSphere$ have the following properties.
\begin{enumerate}[\normalfont(i)]
\item For each $\lambda \in \mathbb{D} \cap \mathbb{R}$, we have $h_\lambda(\mathbb{D})=\mathbb{D}$.
\item For each $\lambda \in \mathbb{D} \cap \mathbb{R}$, the restriction $h_\lambda: \mathbb{T} \to \mathbb{T}$ is a conformal welding of $\mathcal{J}_{\lambda/4}$.
\item The map $h:\mathbb{D} \times \RiemannSphere \to \RiemannSphere$ defined by
$$h(\lambda,z):=h_\lambda(z) \qquad (\lambda \in \mathbb{D}, z \in \RiemannSphere)$$
is a holomorphic motion of $\RiemannSphere$ with $h(\lambda,\infty)=\infty$ for each $\lambda \in \mathbb{D}$.

\end{enumerate}

\end{proposition}

\begin{proof}
For (i), if $\lambda \in \mathbb{D} \cap \mathbb{R}$, then $\overline{\lambda}=\lambda$ and, by construction, the Beltrami coefficient $\mu_\lambda$ is symmetric with respect to $\mathbb{T}$, see \cite[Exercise 1.2.4]{BF}. It easily follows from the uniqueness part of Theorem \ref{MRMT} that the map $h_\lambda$ is symmetric with respect to $\mathbb{T}$, see \cite[Exercise 1.4.1]{BF}. Thus $h_\lambda(\mathbb{T})=\mathbb{T}$, from which we deduce that $h_\lambda(\mathbb{D})=\mathbb{D}$, since $h_\lambda$ fixes $0$.

For (ii), note that the Beltrami coefficients of $h_\lambda$ and $g_\lambda$ are equal almost everywhere on $\mathbb{D}$. By the uniqueness part of Theorem \ref{MRMT}, we have that $f_\lambda:=g_\lambda \circ h_\lambda^{-1}$ is conformal on $h_\lambda(\mathbb{D})=\mathbb{D}$, and $f_\lambda$ maps $\mathbb{D}$ onto $\Omega_{\lambda}^-$. It follows that $h_\lambda= f_\lambda^{-1} \circ g_\lambda$ is a conformal welding of $\mathcal{J}_{\lambda/4}$, as required.

For (iii), note that by Theorem \ref{thmBR}, the function $\lambda \mapsto \nu_\lambda$ is holomorphic on $\mathbb{D}$ and $\|\nu_\lambda\|_\infty \leq |\lambda|$ for all $\lambda \in \mathbb{D}$. Clearly, this remains true if $\nu_\lambda$ is replaced by $\mu_\lambda$. It follows that the map $h:\mathbb{D} \times \RiemannSphere \to \RiemannSphere$ defined by
$$h(\lambda,z)=h_\lambda(z) \qquad (\lambda \in \mathbb{D}, z \in \RiemannSphere)$$
is a holomorphic motion, again by Theorem \ref{thmBR}. Lastly, we have $h(\lambda,\infty)=\infty$ for each $\lambda \in \mathbb{D}$ by construction.
\end{proof}

Now, recall from \S\ref{S:sing} that there exists $\lambda_0$ such that for all $\lambda \in \mathbb{D} \cap \mathbb{R}$ with $0<|\lambda|<\lambda_0$, the harmonic measures $\omega_{\lambda}^-$ and $\omega_{\lambda}^+$ are mutually singular. Here $\omega_{\lambda}^-, \omega_{\lambda}^+$ denote the harmonic measures of the domains $\Omega_{\lambda}^-,\Omega_{\lambda}^+$ with respect to $0,\infty$ respectively.

\begin{lemma}
\label{singular}
For each $\lambda \in \mathbb{D} \cap \mathbb{R}$ with $0<|\lambda|<\lambda_0$, there exists a Borel set $A^\lambda \subset \mathbb{T}$ with $\sigma(A^\lambda)=1$ but $\sigma(h_\lambda(A^\lambda))=0$, where $\sigma$ is Lebesgue measure on $\mathbb{T}$ normalized so that $\sigma(\mathbb{T})=1$.
\end{lemma}

\begin{proof}
Fix $\lambda \in \mathbb{D} \cap \mathbb{R}$ with $0<|\lambda|<\lambda_0$. By part (ii) of Proposition \ref{confweld}, we have that $h_\lambda= f_\lambda^{-1} \circ g_\lambda$, where $f_\lambda:\mathbb{D} \to \Omega_{\lambda}^-$ and $g_\lambda:\RiemannSphere \setminus \overline{\mathbb{D}} \to \Omega_{\lambda}^+$ are conformal. Now, since the measures $\omega_{\lambda}^-$ and $\omega_{\lambda}^+$ on $\mathcal{J}_{\lambda/4}$ are mutually singular, there exists a Borel set $B^\lambda \subset \mathcal{J}_{\lambda/4}$ with $\omega_{\lambda}^+(B^\lambda)=1$ but $\omega_{\lambda}^-(B^\lambda)=0$. Set $A^\lambda:=g_\lambda^{-1}(B^\lambda) \subset \mathbb{T}$. Then $A^\lambda$ is a Borel subset of $\mathbb{T}$ with $\sigma(A^\lambda)=1$, by conformal invariance of harmonic measure (see e.g. \cite[Chapter I]{GAM}). On the other hand, since $\omega_{\lambda}^-(g_\lambda(A^\lambda))=0$, we get that $\sigma(h_\lambda(A^\lambda))=0$, again by conformal invariance of harmonic measure.
\end{proof}

\begin{corollary}
\label{singular2}
For every sequence $(\lambda_n) \subset \mathbb{D} \cap \mathbb{R}$ converging to $0$ with $0<|\lambda_n|<\lambda_0$ for all $n$, there exists a compact set $E \subset \mathbb{T}$ with $\sigma(E)>0$ but $\sigma(h_{\lambda_n}(E))=0$ for all $n$.
\end{corollary}

\begin{proof}
Let $(\lambda_n) \subset \mathbb{D} \cap \mathbb{R}$ be a sequence converging to $0$ with $0<|\lambda_n|<\lambda_0$ for all $n$. Then for each $n$, there exists a Borel set $A^{\lambda_n} \subset \mathbb{T}$ with $\sigma(A^{\lambda_n})=1$ but $\sigma(h_{\lambda_n}(A^{\lambda_n}))=0$, by Lemma \ref{singular2}. Set $A:= \cap_n A^{\lambda_n}$. Then $\sigma(A)=1$ and $\sigma(h_{\lambda_n}(A))=0$ for all $n$. By regularity of Lebesgue measure, there exists a compact set $E \subset A$ with $\sigma(E)>0$, and the result follows.
\end{proof}

\begin{proof}[Completion of the proof of Theorem~\ref{mainthm2}]
Let $h_{\lambda}$ be as in Proposition \ref{confweld}, so that the map $h:\mathbb{D} \times \RiemannSphere \to \RiemannSphere$ defined by
$$h(\lambda,z):=h_{\lambda}(z) \qquad (\lambda \in \mathbb{D}, z \in \RiemannSphere)$$
is a holomorphic motion of $\RiemannSphere$ with $h_\lambda(\infty)=\infty$ for all $\lambda \in \mathbb{D}$.

Now, take any sequence $(\lambda_n) \subset \mathbb{D} \cap \mathbb{R}$ converging to $0$ with $0<|\lambda_n|<\lambda_0$ for all $n$, and let $E \subset \mathbb{T}$ be as in Corollary \ref{singular2}. Then $\sigma(E)>0$, but $\sigma(E_{\lambda_n})=0$ for all $n$. By a result of Murai (see \cite[Corollary 3]{MUR}), the analytic capacity of a compact subset of $\mathbb{T}$ is zero if and only if its length is zero. Thus $\gamma(E)>0$, whereas $\gamma(E_{\lambda_n})=0$ for all $n$, and the function
$$\lambda \mapsto \gamma(E_\lambda) \qquad (\lambda \in \mathbb{D})$$
is discontinuous at $\lambda=0$, as required.
\end{proof}

\section{Concluding remarks}\label{S:conclusion}

(1) Examples such as the one constructed in Theorem~\ref{mainthm2} do not exist if, in addition,
 the maps $h_\lambda$ are required to be conformal on $\RiemannSphere \setminus E$. Indeed, by \cite[Theorem 1.1]{PRY}, this additional assumption on $h$ implies that the function
$$\lambda \mapsto \gamma(E_\lambda) \qquad (\lambda \in \mathbb{D})$$
is $C^{\infty}$, since its logarithm is harmonic. The B\"{o}ttcher motion $g$ of \S\ref{S:Julia} is an example of such a holomorphic motion.

(2) As mentioned in the introduction,
analytic capacity is closely related to one-dimensional Hausdorff measure $\mathcal {H}^1$.
This suggests an alternative, maybe simpler, approach to proving Theorem~\ref{mainthm2}, namely, to study
the continuity of $\mathcal{H}^1(E_\lambda)$.

Once again, let $\mathcal{J}_c$ be the Julia set of $z^2+c$
and let $\mathcal{M}$ denote the Mandelbrot set.
This time, we are interested in the case where $c\notin\mathcal{M}$.
It is well known that  $\dim_H(\mathcal{J}_c)\to0$ as $|c|\to\infty$,
and also that $\sup_{c\notin\mathcal{M}}\dim_H(\mathcal{J}_c)=2$.
Also, using Theorem~\ref{Astala}, one sees that the function $c\mapsto\dim_H(\mathcal{J}_c)$
is continuous on $\mathbb{C}\setminus\mathcal{M}$.
Therefore there exists $c_0\in \mathbb{C}\setminus\mathcal{M}$ and a sequence $c_n$ converging to $c_0$
such that $\dim_H(\mathcal{J}_{c_n})<1$ for all $n$ but $\dim_H(\mathcal{J}_{c_0})=1$.
Clearly $\mathcal{H}^1(\mathcal{J}_{c_n})=0$ for all~$n$. On the other hand,
by the general theory of hyperbolic rational functions (see e.g.\ \cite{URB}),
$0<\mathcal{H}^1(\mathcal{J}_{c_0})<\infty$.
Thus the function $c\mapsto\mathcal{H}^1(\mathcal{J}_{c})$ is discontinuous at $c_0$.

This example can easily be converted into an example of a holomorphic motion $E_\lambda$ for which
$\lambda\mapsto\mathcal{H}^1(E_\lambda)$  is discontinuous: just take $E_\lambda:=\mathcal{J}_{c_0+a\lambda}$ for a small enough value of $a>0$. Moreover, the analogous construction works for $s$-dimensional Hausdorff measure, for each $s\in(0,2)$. On the other hand, when $s=2$ the corresponding Hausdorff measure is just a multiple of  area measure, and, as we have seen in Theorem~\ref{Astala}, the area of a holomorphic motion \emph{is} continuous.

The construction just described (for $\mathcal{H}^1$) comes close to giving another example that would prove Theorem~\ref{mainthm2}. Indeed, since $\mathcal{H}^1$-measure zero implies analytic capacity zero, the sets $\mathcal{J}_{c_n}$ all have analytic capacity zero. All that is lacking is to show that $\mathcal{J}_{c_0}$ is of positive analytic capacity. However, as mentioned in the introduction, positive $\mathcal{H}^1$-measure does not in general imply positive analytic capacity. The standard counter-example is the so-called Cantor $1/4$-square
(see e.g.\ \cite[Chapter~IV, Theorem~2.7]{GAR}).
As the set $\mathcal{J}_{c_0}$ is also a Cantor-type set,
it is unclear whether its analytic capacity is positive or not. So we conclude by formally posing this as question.

\begin{question}
Is $\gamma(\mathcal{J}_{c_0})>0$?
\end{question}

\subsection*{Acknowledgments}
The authors would like to thank Peter Lin for fruitful conversations.

\bibliographystyle{amsplain}

\begin{thebibliography}{99}

\bibitem{AHL}
L. Ahlfors,
Bounded analytic functions,
\textsl{Duke Math. J.},
\textbf{14} (1947),
1--11.

\bibitem{AST}
K. Astala,
Area distortion of quasiconformal mappings,
\textit{Acta Math.},
\textbf{173} (1994),
37--60.

\bibitem{AIM}
K. Astala, T. Iwaniec, G.J. Martin,
\textsl{Elliptic partial differential equations and quasiconformal mappings in the plane},
Princeton University Press, Princeton, NJ,
2009.

\bibitem{AM}
K. Astala, G. J. Martin,
Holomorphic motions, Papers on analysis,
\textit{Rep. Univ. Jyv\"askyl\"a Dep. Math. Stat.} vol. 83,
Univ. Jyv\"askyl\"a, 2001, 27--40.

\bibitem{AZ}
K. Astala, M. Zinsmeister,
Holomorphic families of quasi-Fuchsian groups,
\textit{Ergodic Theory Dynam. Systems},
\textbf{14} (1994),
207--212.

\bibitem{BR}
L. Bers and H. Royden,
Holomorphic families of injections,
\textsl{Acta Math.},
\textbf{157} (1986),
259--286.

\bibitem{BCGJ}
C.J. Bishop, L. Carleson, J.B. Garnett, P.W. Jones,
Harmonic measures supported on curves,
\textsl{Pacific J. Math.},
\textbf{138} (1989),
233--236.

\bibitem{BF}
B. Branner, N. Fagella,
\textsl{Quasiconformal surgery in holomorphic dynamics},
Cambridge University Press, Cambridge,
2014.

\bibitem{EM}
C.J. Earle and S. Mitra,
Variation of moduli under holomorphic motions,
\textsl{Contemp. Math.},
\textbf{256} (2000),
39--67.

\bibitem{GAR}
J. Garnett,
\textsl{Analytic capacity and measure},
Springer-Verlag, Berlin,
1972.

\bibitem{GAM}
J. Garnett, D.E. Marshall,
\textsl{Harmonic measure},
Cambridge University Press, Cambridge,
2005.

\bibitem{HAM}
D.H. Hamilton,
Length of Julia curves,
\textsl{Pacific J. Math.},
\textbf{169} (1995),
75--93.

\bibitem{LEH}
O. Lehto and K.I. Virtanen,
\textsl{Quasiconformal mappings in the plane},
Springer-Verlag, New York-Heidelberg,
1973.


\bibitem{MSS}
R. Ma\~{n}\'{e}, P. Sad, D. Sullivan,
On the dynamics of rational maps,
\textsl{Ann. Sci. \'{E}cole Norm. Sup. (4)},
\textbf{16} (1983),
193--217.

\bibitem{MCM}
C.T. McMullen,
\textit{Complex dynamics and renormalization},
Princeton University Press, Princeton, NJ,
1994.

\bibitem{MIL}
J. Milnor,
\textit{Dynamics in one complex variable},
Princeton University Press, Princeton, NJ,
2006.

\bibitem{MUR}
T. Murai,
A formula for analytic separation capacity,
\textsl{Kodai Math. J.},
\textbf{13} (1990),
265--288.

\bibitem{PRY}
S. Pouliasis, T. Ransford, M. Younsi,
Analytic capacity and holomorphic motions,
\textsl{Conform. Geom. Dyn.},
\textbf{23} (2019),
130--134.

\bibitem{RAN2}
T. Ransford,
\textsl{Potential theory in the complex plane},
Cambridge University Press, Cambridge,
1995.

\bibitem{RAN}
T.J. Ransford,
Variation of Hausdorff dimension of Julia sets,
\textsl{Ergodic Theory Dynam. Systems},
\textbf{13} (1993),
167--174.


\bibitem{RUE}
D. Ruelle,
Repellers for real analytic maps,
\textsl{Ergodic Theory Dynam. Systems},
\textbf{2} (1982),
99--107.

\bibitem{SLO}
Z. S\l odkowski,
Holomorphic motions and polynomial hulls,
\textsl{Proc. Amer. Math. Soc.},
\textbf{111} (1991),
347--355.

\bibitem{STE}
N. Steinmetz,
\textsl{Rational iteration},
Walter de Gruyter \& Co., Berlin,
1993.

\bibitem{ST}
D. Sullivan, W. P. Thurston,
Extending holomorphic motions,
\textsl{Acta Math.},
\textbf{157} (1986),
243--257.

\bibitem{TOL}
X. Tolsa,
\textsl{Analytic Capacity, the Cauchy Transform, and Non-homogeneous Calder\'on--Zygmund Theory},
Birkh\"auser/Springer, Cham,
2014.

\bibitem{URB}
M. Urba\'nski,
Measures and dimensions in conformal dynamics,
\textsl{Bull. Amer. Math. Soc. (N.S.)},
\textbf{40} (2003),
281--321.

\bibitem{VIT}
A. Vitu{\v{s}}kin,
Analytic capacity of sets in problems of approximation theory (Russian),
\textsl{Uspehi Mat. Nauk},
\textbf{22} (1967),
141--199.

\bibitem{YA}
H. Yamaguchi,
Sur une uniformit\'e des surfaces constantes d'une fonction enti\`ere de deux variables complexes,
\textsl{J. Math. Kyoto Univ.},
\textbf{13} (1973),
 417--433.

\bibitem{YOU}
M. Younsi,
Analytic capacity: computation and related problems,
\textsl{Theta Ser. Adv. Math.},
\textbf{22} (2018),
121--152.

\bibitem{YOR}
M. Younsi, T. Ransford,
Computation of analytic capacity and applications to the subadditivity problem,
\textsl{Comput. Methods Funct. Theory},
\textbf{13} (2013),
337--382.

\bibitem{ZAK}
S. Zakeri,
Conformal fitness and uniformization of holomorphically moving disks,
\textsl{Trans. Amer. Math. Soc.}
\textbf{368} (2016),
1023--1049.






\end{thebibliography}

\end{document}